\documentclass{amsart}%
\usepackage{amssymb}
\usepackage{amsmath}
\usepackage{amsfonts}
\usepackage{graphicx}%
\newtheorem{theorem}{Theorem}
\theoremstyle{plain}

\newtheorem{corollary}{Corollary}

\newtheorem{definition}{Definition}

\newtheorem{proposition}{Proposition}
\newtheorem{remark}{Remark}

\numberwithin{equation}{section}
\begin{document}
\title[Spectral measures on C-algebras]{Spectral measures on C-algebras of operators in $c_{0}(\mathbb{N})$}
\author{Aguayo, J.}
\curraddr{{\small Departamento de Matem\'{a}tica, Facultad de Ciencias F\'{\i}sicas y
Matem\'{a}ticas, Universidad de Concepcion, Casilla 160-C,
Concepci\'{o}n-Chile.}}
\email{jaguayo@udec.cl}
\author{Nova, M.}
\curraddr{{\small Departamento de Matem\'{a}tica y F\'{\i}sica Aplicadas, Facultad de
Ingenier\'{\i}a, Universidad Cat\'{o}lica de la Sant\'{\i}sima Concepci\'{o}n,
Casilla 297, Concepci\'{o}n Chile,}}
\email{{\small mnova@ucsc.cl}}
\author{Ojeda, J}
\curraddr{{\small Departamento de Matem\'{a}tica, Facultad de Ciencias F\'{\i}sicas y
Matem\'{a}ticas, Universidad de Concepcion, Casilla 160-C,
Concepci\'{o}n-Chile.}}
\email{jacqojeda@udec.cl}
\thanks{This work was partially supported by Proyecto VRID N$^{\circ}$ 214.014.038-1.0IN}

\begin{abstract}
The main goal of this work is to introduce an analogous in the non-archimedean
context of the Gelfand spaces of some commutative Banach algebras with unit.
In order to do that, we will study the spectrum of this algebras and will show
that, under special condition, these algebras are isometrically isomorphic to
a respective spaces of continuous functions defined over some compacts. Such
isometries will preserve projections and will allow us to define associated
measures which are known like spectral measures. We will finish this work by
showing that any element of these algebras are integrable under these measures.

\end{abstract}
\maketitle

\section{Introduction and notation}

Many researchers have tried to generalize the elemental studies of Banach
algebras from classical case to vectorial structures over non-archimedean
fields. The first big task was to find a results similar to the Gelfand-Mazur
Theorem in this context. But, this theorem failed since every field $\Bbbk$
with a non-archimedean valuation is contained in another field $\mathbb{K}$
such that its valuation is an extension of the valuation of $\Bbbk$.

One of the main pioneers in the study of non-archimedean Banach algebras of
linear operators and spectral theory in this context has been M. Vishik
$\left[  7\right]  $, especially in the class of linear operators which admit
compact spectrum. We can also mention another important pioneer, Berkovick
$\left[  3\right]  ,$ who made a deep study of this subject on his survey.

The main goal of this work is to introduce an analogous in the non-archimedean
context of the Gelfand spaces of some commutative Banach algebras of linear
operators with unit. In order to do that, we will study the spectrum of this
algebras and will show that, under special condition, these algebras are
isometrically isomorphic to a respective spaces of continuous functions
defined over some compact. Such isometries will preserve idempotent elements
and will allow us to define associated measures which are known as spectral
measures. We will finish this work showing that each element of the
commutative Banach algebra described before can be represented as an integral
of some continuous function, where the integral has been defined by the
spectral measure.

Throughout this paper $\mathbb{K}$ is a valued field which is complete with
respect to the metric induced by the nontrivial non-archimedean valuation
$\left\vert \cdot\right\vert $ and its residue class field is formally real.

In the classical situation we can distinguish two type of normed spaces: those
spaces which are separable and those which are not separable. If $E$ is a
separable normed space over $\mathbb{K}$, then each one-dimensional subspaces
of $E$ is homeomorphic to $\mathbb{K}$, so $\mathbb{K}$ must be separable too.
Nevertheless, we know that there exist non-archimedean fields which are not
separable. Thus, for non-archimedean normed spaces the concept of separability
is meaningless if $\mathbb{K}$ is not separable. However, linearizing the
notion of separability, we obtain a useful generalization of this concept. A
normed space $E$ over a non-archimedean valued field is said to be
of\ countable\ type if it contains a countable subset whose linear hull is
dense in $E$. An example of a normed space of countable type is $\left(
c_{0},\left\Vert \cdot\right\Vert _{\infty}\right)  ,$ where $c_{0}$ is the
Banach space of all sequences $x=\left(  a_{n}\right)  _{n\in\mathbb{N}}$,
$a_{n}\in\mathbb{K}$, for which $\lim_{n\rightarrow\infty}a_{n}=0$ and its
norm is given by $\left\Vert x\right\Vert _{\infty}=\sup\left\{  \left\vert
a_{n}\right\vert :n\in\mathbb{N}\right\}  .$

A non-archimedean Banach space $E$ is said to be Free Banach space if there
exists a family $\left\{  e_{i}\right\}  _{i\in I}$ of non-null vectors of $E$
such that any element $x$ of $E$ can be written in the form of convergent sum
$x=\sum_{i\in I}x_{i}e_{i},\ x_{i}\in\mathbb{K},$ and $\left\Vert x\right\Vert
=\sup_{i\in I}\left\vert x_{i}\right\vert \left\Vert e_{i}\right\Vert .$ The
family $\left\{  e_{i}\right\}  _{i\in I}$ is called orthogonal basis of $E.$
If $s:I\rightarrow\left(  0,\infty\right)  ,$ then an example of Free Banach
space is $c_{0}\left(  I,\mathbb{K},s\right)  ,$ the collection of all
$x=\left(  x_{i}\right)  _{i\in I}$ such that for any $\epsilon>0,$ the set
$\left\{  i\in I:\left\vert x_{i}\right\vert s\left(  i\right)  >\epsilon
\right\}  $ is, at most, finite (or, equivalently, $\lim_{i\in I}\left\vert
x_{i}\right\vert s\left(  i\right)  =0,$ with respect to the Frechet filter on
$I$) and $\left\Vert x\right\Vert =\sup_{i\in I}\left\vert x_{i}\right\vert
s\left(  i\right)  $.

We already know that a Free Banach space $E$ is isometrically isomorphic to
$c_{0}\left(  I,\mathbb{K},s\right)  ,$ for some $s:I\rightarrow\left(
0,\infty\right)  .$ In particular if a Free Banach space is of countable type,
then it is isometrically isomorphic to $c_{0}\left(  \mathbb{N},\mathbb{K}%
,s\right)  ,$ for some $s:\mathbb{N}\rightarrow\left(  0,\infty\right)  .$
Note that if $s\left(  i\right)  \in\left\vert \mathbb{K}\right\vert ,$ then
$E$ is isometrically isomorphic to $c_{0}\left(  \mathbb{N},\mathbb{K}\right)
$ (or $c_{0}$ in short). \ For a detailed study of Free Banach spaces, in
general, we refer the reader to $\left[  4\right]  $

Now, since residual class field of $\mathbb{K}$ is formally real, the bilinear
form%
\[
\left\langle \cdot,\cdot\right\rangle :c_{0}\times c_{0}\rightarrow
\mathbb{K};\ \left\langle x,y\right\rangle =\sum_{i=1}^{\infty}x_{i}y_{i}%
\]
is an inner product, $\left\Vert \cdot\right\Vert =\sqrt{\left\vert
\left\langle \cdot,\cdot\right\rangle \right\vert }$ is a norm in $c_{0}$ and
the supremum norm $\left\Vert \cdot\right\Vert _{\infty}$ coincides with
$\left\Vert \cdot\right\Vert $, that is, $\left\Vert \cdot\right\Vert
=\left\Vert \cdot\right\Vert _{\infty}$ (see $\left[  5\right]  $). Therefore,
to study the Free Banach spaces of countable type it is enough to study the
space $c_{0}.$

If $E$ and $F$ are $\mathbb{K}$-normed spaces, then $\mathcal{L}\left(
E,F\right)  $ will be the $\mathbb{K}$-normed space consisting of all
continuous linear maps from $E$ into $F.$ If $F=E,$ then $\mathcal{L}\left(
E\right)  =\mathcal{L}\left(  E,E\right)  .$ For any $T\in\mathcal{L}\left(
E,F\right)  ,$ $N\left(  T\right)  $ will denote its Kernel and $R\left(
T\right)  $ its range.

A linear operator $T$ from $E$ into $F$ is said to be compact operator if
$T\left(  B_{E}\right)  $ is compactoid, where $B_{E}$ =$\left\{  x\in
E:\left\Vert x\right\Vert \leq1\right\}  $ is the unit ball of $E.$ It was
proved in $\left[  6\right]  $ that $T$ is compact if and only if, for each
$\epsilon>0,$ there exists a lineal operator of finite-dimensional range $S$
such that $\left\Vert T-S\right\Vert \leq\epsilon.$

Since $c_{0}$ is not orthomodular, there exist operators in $\mathcal{L}%
\left(  c_{0}\right)  $ which do not admit adjoint; for example, $T\left(
x\right)  =\left(  \sum_{i=1}^{\infty}x_{i}\right)  e_{1},\ x=\left(
x_{i}\right)  _{i\in\mathbb{N}}\in c_{0}.$ We will denote by $\mathcal{A}_{0}$
the collection of all elements of $\mathcal{L}\left(  c_{0}\right)  $ which
admit\ adjoint. A characterization of the elements of $\mathcal{A}_{0}$ (see
$\left[  2\right]  $) is the following:%
\[
\mathcal{A}_{0}=\left\{  T\in\mathcal{L}\left(  c_{0}\right)  :\forall y\in
c_{0},\ \lim_{i\rightarrow\infty}\left\langle Te_{i},y\right\rangle
=0\right\}  .
\]
Of course, $\mathcal{A}_{0}$ is a non-commutative Banach algebra with unit. \ 

Now, for each $a=\left(  a_{i}\right)  _{i\in\mathbb{N}}\in c_{0},$ the linear
operator $M_{a},$ defined by $M_{a}\left(  \cdot\right)  =\sum_{i=1}^{\infty
}a_{i}\left\langle \cdot,e_{i}\right\rangle e_{i},$ belongs to $\mathcal{A}%
_{0};$ moreover,
\[
\lim_{n\rightarrow\infty}\left\Vert M_{a}e_{n}\right\Vert =\lim_{n\rightarrow
\infty}\left\Vert \sum_{i=1}^{\infty}a_{i}\left\langle e_{n},e_{i}%
\right\rangle e_{i}\right\Vert =\lim_{n\rightarrow\infty}\left\vert
a_{n}\right\vert =0,
\]
meanwhile, the identity map $I$ is also an element of $\mathcal{A}_{0}$, but%
\[
\lim_{n\rightarrow\infty}\left\Vert I\left(  e_{n}\right)  \right\Vert
=\lim_{n\rightarrow\infty}\left\Vert e_{n}\right\Vert =1.
\]

Let us denote by $\mathcal{A}_{1}$ the collection of all $T\in\mathcal{L}%
\left(  c_{0}\right)  $ such that $\lim_{n\rightarrow\infty}Te_{n}=\theta,$
i.e.,
\[
\mathcal{A}_{1}=\left\{  T\in\mathcal{L}\left(  c_{0}\right)  :\lim
_{n\rightarrow\infty}Te_{n}=\theta\right\}  .
\]
From the fact that
\[
\left\vert \left\langle Te_{n},y\right\rangle \right\vert \leq\left\Vert
Te_{n}\right\Vert \left\Vert y\right\Vert ,
\]
we have that $\mathcal{A}_{1}\subsetneqq\mathcal{A}_{0}$ since $I\notin
\mathcal{A}_{1}.$

By $\left[  4\right]  ,$ we know that each $T\in\mathcal{L}\left(
c_{0}\right)  $ can be represented by $T=\sum_{i,j=1}^{\infty}\alpha
_{i,j}e_{j}^{\prime}\otimes e_{i},$ where $\lim_{i\rightarrow\infty}%
\alpha_{i,j}=0,\ $for all $j\in\mathbb{N}$, $\left\Vert T\right\Vert
=\sup\left\{  \left\Vert T\left(  e_{i}\right)  \right\Vert :i\in
\mathbb{N}\right\}  $ and $T$ is compact if and only if%
\[
\lim_{j\rightarrow\infty}\sup\left\{  \left\vert \alpha_{i,j}\right\vert
:i\in\mathbb{N}\right\}  =0.
\]
Now, since%
\begin{align*}
\left\Vert Te_{n}\right\Vert  &  =\left\Vert \left(  \sum_{i,j=1}^{\infty
}\alpha_{i,j}e_{j}^{\prime}\otimes e_{i}\right)  \left(  e_{n}\right)
\right\Vert =\left\Vert \sum_{i,j=1}^{\infty}\alpha_{i,j}e_{j}^{\prime}\left(
e_{n}\right)  e_{i}\right\Vert \\
&  =\left\Vert \sum_{i=1}^{\infty}\alpha_{i,n}e_{i}\right\Vert =\sup\left\{
\left\vert \alpha_{i,n}\right\vert :i\in\mathbb{N}\right\}  .
\end{align*}
Therefore,%
\[
T\in\mathcal{A}_{1}\Leftrightarrow T\in\mathcal{A}_{0}\text{ and }T\text{ is
compact}%
\]

Let $\left\{  y^{\left(  i\right)  }\right\}  _{i\in\mathbb{N}}$ be a sequence
in $c_{0}.$ We will say that $\left\{  y^{\left(  i\right)  }\right\}
_{i\in\mathbb{N}}$ is orthonormal if $\left\langle y^{\left(  i\right)
},y^{\left(  j\right)  }\right\rangle =0,\ i\neq j,$ and $\left\Vert
y^{\left(  i\right)  }\right\Vert =1.$On the other hand, we will understand by
a normal projection to any projection $P:c_{0}\rightarrow c_{0}$ such that
$\left\langle x,y\right\rangle =0$ for any $x\in N\left(  P\right)  $ and
$y\in R\left(  P\right)  .$ For example, if $y\in c_{0},\ y\neq0,$ is fixed$,$
then $P\left(  \cdot\right)  =\frac{\left\langle \cdot,y\right\rangle
}{\left\langle y,y\right\rangle }y$ is a normal projection.

The next theorem characterizes compact and self-adjoint operators. Its proof
is similar to the proof given in $\left[  2\right]  $, so we will omit it.

\begin{theorem}
If the linear operator $T:c_{0}\rightarrow c_{0}$ is compact and self-adjoint,
then there exists an element $\lambda=\left(  \lambda_{i}\right)
_{i\in\mathbb{N}}\in c_{0}$ and an orthonormal sequence $\left\{  y^{\left(
i\right)  }\right\}  _{i\in\mathbb{N}}$ in $c_{0}$ such that%
\[
T=\sum_{i=1}^{\infty}\lambda_{i}P_{i}%
\]
and $\left\Vert T\right\Vert =\left\Vert \lambda\right\Vert ,$ where
\[
P_{i}\left(  \cdot\right)  =\frac{\left\langle \cdot,y^{\left(  i\right)
}\right\rangle }{\left\langle y^{\left(  i\right)  },y^{\left(  i\right)
}\right\rangle }y^{\left(  i\right)  }%
\]
is the normal projection defined by $y^{\left(  i\right)  }.$
\end{theorem}

\begin{remark}
This theorem gives us a characterization for compact and self-adjoint
operators. In fact, it is not hard to see that if we take $\lambda=\left(
\lambda_{i}\right)  _{i\in\mathbb{N}}\in c_{0}\ $and an orthonormal sequences
$\left\{  y^{\left(  i\right)  }\right\}  _{i\in\mathbb{N}}$ in $c_{0}$, the
operator
\[
T=\sum_{i=1}^{\infty}\lambda_{i}P_{i},
\]
where $P_{i}$ is as in the Theorem$,$ is compact and self-adjoint and
$\left\Vert T\right\Vert =\left\Vert \lambda\right\Vert $.
\end{remark}

\section{Algebra of operators}

\subsection{An algebra without unit}

From now on, we will consider a fixed orthonormal sequence $Y=\left\{
y^{\left(  i\right)  }\right\}  _{i\in\mathbb{N}}$ in $c_{0}.$ We will denote
by $\mathfrak{T}_{Y}(c_{0})$ the collection of all compact operators
$T_{\lambda},\ \lambda\in c_{0},$ where%
\[
T_{\lambda}=\sum_{i=1}^{\infty}\lambda_{i}P_{i}%
\]
As we know, the adjoint $T_{\lambda}^{\ast}$ of $T_{\lambda}$ is itself and
$\lim_{n\rightarrow\infty}T_{\lambda}\left(  e_{i}\right)  =0.$ Obviously, the
collection $\mathfrak{T}_{Y}(c_{0})\,\ $is a linear space, since%
\[
T_{\lambda}+T_{\mu}=T_{\lambda+\mu};\ \ \alpha T_{\lambda}=T_{\alpha\lambda}%
\]
On the other hand, since $c_{0}$ is a commutative algebra with the operation
$\lambda\cdot\mu=\left(  \lambda_{i}\mu_{i}\right)  ,$ we have%
\[
T_{\lambda}\circ T_{\mu}=T_{\lambda\cdot\mu}=T_{\mu}\circ T_{\lambda}.
\]
Therefore, $\mathfrak{T}_{Y}(c_{0})$ is a commutative algebra without unit.
Even more, by the fact that $T_{\lambda}=T_{\mu}$ implies $\lambda=\mu$ (see
$\left[  2\right]  $), the linear transformation%
\[
\Lambda:c_{0}\rightarrow\mathfrak{T}_{Y}(c_{0});\ \lambda\rightarrow
\Lambda\left(  \lambda\right)  =T_{\lambda}%
\]
is an isometric isomorphism of algebras.

As we know, each algebra without unit can be transformed in an algebra with
unit, considering the collection $E^{+}=\mathbb{K}\oplus E$ provided with the
usual linear operations and the multiplication operation defined by%
\[
\left(  \alpha,\mu\right)  \cdot\left(  \beta,\nu\right)  =\left(  \alpha
\beta,\alpha\nu+\beta\mu+\mu\cdot\nu\right)
\]
The unit of this algebra is $\left(  1,\theta\right)  $. If $E$ is, in
particular, a normed space, then $E^{+}$ so is and%
\[
\left\Vert \left(  \alpha,\mu\right)  \right\Vert =\max\left\{  \left\vert
\alpha\right\vert ,\left\Vert \mu\right\Vert \right\}  .
\]

Now, the commutative Banach algebra $\left(  \mathfrak{T}_{Y}(c_{0}%
),+,\cdot,\circ,\left\Vert \cdot\right\Vert \right)  $ can be transformed, as
above, in a commutative Banach algebra $\left(  \mathfrak{T}_{Y}(c_{0}%
)^{+},+,\cdot,\circ,\left\Vert \cdot\right\Vert \right)  $ with unit. By the
fact that $c_{0}$ is isometrically isomorphic to $\mathfrak{T}_{Y}(c_{0})$,
$\mathfrak{T}_{Y}(c_{0})^{+}$ is isometrically isomorphic to $c_{0}{}^{+}$

\subsection{An algebra with unit}

We will denote by $\mathcal{S}_{Y}\left(  c_{0}\right)  $ the collection of
all linear operators $\alpha I+T_{\lambda},\ \alpha\in\mathbb{K}$
and$\ T_{\lambda}\in\mathfrak{T}_{Y}(c_{0}).$ $\mathcal{S}_{Y}\left(
c_{0}\right)  $ is an normed space and since
\begin{align*}
\left(  \alpha_{1}I+T_{\mu}\right)  \circ\left(  \alpha_{2}I+T_{\nu}\right)
&  =\alpha_{1}\alpha_{2}I+\alpha_{1}T_{\nu}+\alpha_{2}T_{\mu}+T_{\mu}\circ
T_{\nu}\\
&  =\alpha_{1}\alpha_{2}I+T_{\alpha_{1}\nu+\alpha_{2}\mu+\mu\nu}%
\end{align*}
we conclude that $\mathcal{S}_{Y}\left(  c_{0}\right)  $ is a commutative
algebra with unit.

\begin{theorem}
The algebra $\mathcal{S}_{Y}\left(  c_{0}\right)  $ is isometrically
isomorphic to $\mathfrak{T}_{Y}(c_{0})^{+}.$ As a consequence, $\mathcal{S}%
_{Y}\left(  c_{0}\right)  $ is a commutative Banach algebra with unit.

\begin{proof}
We define%
\begin{align*}
\mathfrak{T}_{Y}(c_{0})^{+}  &  \rightarrow\mathcal{S}_{Y}\left(  c_{0}\right)
\\
\left(  \alpha,T_{\lambda}\right)   &  \rightarrow\alpha I+T_{\lambda}%
\end{align*}
Since $\alpha T_{\mu}+\beta T_{\lambda}+T_{\lambda}\circ T_{\mu}=T_{\alpha
\mu+\beta\lambda+\mu\lambda},$ this transformation is an algebra
homomorphism$.$ Obviously, this homomorphism is onto; hence it is enough to
prove that it is an isometry. We claim that%
\[
\left\Vert \alpha I+T_{\lambda}\right\Vert =\left\Vert \left(  \alpha
,T_{\lambda}\right)  \right\Vert =\max\left\{  \left\vert \alpha\right\vert
,\left\Vert T_{\lambda}\right\Vert \right\}
\]
If $\alpha=0$ or $\left\Vert T_{\lambda}\right\Vert =0$ or $\left\Vert \alpha
I\right\Vert \neq\left\Vert T_{\lambda}\right\Vert ,$ we are done. We only
need to check when%
\[
\left\vert \alpha\right\vert =\left\Vert \alpha I\right\Vert =\left\Vert
T_{\lambda}\right\Vert \neq0.
\]
Of course,
\[
\left\Vert \alpha I+T_{\lambda}\right\Vert \leq\max\left\{  \left\vert
\alpha\right\vert ,\left\Vert T_{\lambda}\right\Vert \right\}  .
\]
Now, by the compactness of $T_{\lambda}$,
\[
\lim_{n\rightarrow\infty}T_{\lambda}\left(  e_{n}\right)  =0.
\]
Thus, there exists $N\in\mathbb{N}$ such that
\[
n\geq N\Rightarrow\left\Vert T_{\lambda}\left(  e_{n}\right)  \right\Vert
<\left\vert \alpha\right\vert
\]
Therefore,%
\begin{align*}
\left\Vert \alpha I+T_{\lambda}\right\Vert  &  =\sup\left\{  \left\Vert \alpha
e_{n}+T_{\lambda}\left(  e_{n}\right)  \right\Vert :n\in\mathbb{N}\right\} \\
&  =\max\left\{  \left\Vert \alpha e_{1}+T_{\lambda}\left(  e_{1}\right)
\right\Vert ,\ \left\Vert \alpha e_{2}+T_{\lambda}\left(  e_{2}\right)
\right\Vert ,\ldots,\left\Vert \alpha e_{N-1}+T_{\lambda}\left(
e_{N-1}\right)  \right\Vert ,\left\vert \alpha\right\vert \right\} \\
&  =\left\vert \alpha\right\vert =\max\left\{  \left\vert \alpha\right\vert
,\left\Vert T_{\lambda}\right\Vert \right\}  =\left\Vert \left(
\alpha,T_{\lambda}\right)  \right\Vert .
\end{align*}

\end{proof}
\end{theorem}

\begin{remark}
Since $c_{0}^{+}$ is isometrically isomorphic to $\mathfrak{T}_{Y}(c_{0})^{+}$
and, at the same time, this last is isometrically isomorphic to $\mathcal{S}%
_{Y}\left(  c_{0}\right)  ,$ $c_{0}^{+}$ is isometrically isomorphic to
$\mathcal{S}_{Y}\left(  c_{0}\right)  .$
\end{remark}

We claim that the usual norm in $\mathcal{S}_{Y}\left(  c_{0}\right)  $ is
power multiplicative, that is, for each $T\in\mathcal{S}_{Y}\left(
c_{0}\right)  ,$
\[
\left\Vert T^{n}\right\Vert =\left\Vert T\right\Vert ^{n}.
\]
In fact, by the remark, it is enough to study this property in $c_{0}^{+}.$

The norm in $c_{0}^{+}$ was defined by $\left\Vert \left(  \alpha,a\right)
\right\Vert =\max\left\{  \left\vert \alpha\right\vert ,\left\Vert
a\right\Vert \right\}  $. On the other hand, for any $\left(  \alpha,a\right)
\in c_{0}^{+}$
\begin{align*}
\left(  \alpha,a\right)  ^{2}  &  =\left(  \alpha^{2},2\alpha a+a^{2}\right)
\\
\left(  \alpha,a\right)  ^{3}  &  =\left(  \alpha^{2},2\alpha a+a^{2}\right)
\cdot\left(  \alpha,a\right) \\
&  =\left(  \alpha^{3},\left(
\begin{array}
[c]{c}%
3\\
1
\end{array}
\right)  \alpha^{2}a+\left(
\begin{array}
[c]{c}%
3\\
2
\end{array}
\right)  \alpha a^{2}+\left(
\begin{array}
[c]{c}%
3\\
3
\end{array}
\right)  a^{3}\right)
\end{align*}
In general, for each $k\in\mathbb{N},$%
\[
\left(  \alpha,a\right)  ^{k}=\left(  \alpha^{k},\sum_{i=1}^{k}\left(
\begin{array}
[c]{c}%
k\\
i
\end{array}
\right)  \alpha^{k-i}a^{i}\right)
\]
Now, let us start analyzing $\left\Vert \left(  \alpha,a\right)
^{k}\right\Vert .$ In general, we have%
\[
\left\Vert \left(  \alpha,a\right)  ^{k}\right\Vert \leq\left\Vert \left(
\alpha,a\right)  \right\Vert ^{k}%
\]
The power multiplicative property is, obviously, satisfied by the supremum
norm in $c_{0},$ that is,
\[
\left\Vert a^{k}\right\Vert =\left\Vert a\right\Vert ^{k}.
\]
\newline If $\alpha=0$ or $a=\theta,$ then
\[
\left\Vert \left(  \alpha,a\right)  ^{k}\right\Vert =\left\Vert \left(
\alpha,a\right)  \right\Vert ^{k}%
\]
If $\left\vert \alpha\right\vert <\left\Vert a\right\Vert $, then we have%
\begin{align*}
\left\Vert \sum_{i=1}^{k-1}\left(
\begin{array}
[c]{c}%
k\\
i
\end{array}
\right)  \alpha^{k-i}a^{i}\right\Vert  &  \leq\max\left\{  \left\vert
\alpha\right\vert ^{k-i}\left\Vert a^{i}\right\Vert :i=1,\ldots,k-1\right\} \\
&  <\max\left\{  \left\Vert a\right\Vert ^{k-i}\left\Vert a\right\Vert
^{i}:i=1,\ldots,k-1\right\}  =\left\Vert a\right\Vert ^{k}%
\end{align*}
Thus,%
\[
\left\Vert \sum_{i=1}^{k}\left(
\begin{array}
[c]{c}%
k\\
i
\end{array}
\right)  \alpha^{k-i}a^{i}\right\Vert =\left\Vert a^{k}+\sum_{i=1}%
^{k-1}\left(
\begin{array}
[c]{c}%
k\\
i
\end{array}
\right)  \alpha^{k-i}a^{i}\right\Vert =\left\Vert a\right\Vert ^{k}%
\]
Therefore,%
\begin{align*}
\left\Vert \left(  \alpha,a\right)  \right\Vert ^{k}  &  =\max\left\{
\left\vert \alpha\right\vert ^{k},\left\Vert a\right\Vert ^{k}\right\} \\
&  =\left\Vert \left(  \alpha,a\right)  ^{k}\right\Vert
\end{align*}
Suppose, finally, that $\left\vert \alpha\right\vert \geq\left\Vert
a\right\Vert .$ Then,%
\[
\left\Vert \sum_{i=1}^{k}\left(
\begin{array}
[c]{c}%
k\\
i
\end{array}
\right)  \alpha^{k-i}a^{i}\right\Vert \leq\max\left\{  \left\vert
\alpha\right\vert ^{k-i}\left\Vert a^{i}\right\Vert :i=1,\ldots,k\right\}
\leq\left\vert \alpha\right\vert ^{k}%
\]
which implies
\begin{align*}
\left\Vert \left(  \alpha,a\right)  ^{k}\right\Vert  &  =\max\left\{
\left\vert \alpha\right\vert ^{k},\left\Vert \sum_{i=1}^{k}\left(
\begin{array}
[c]{c}%
k\\
i
\end{array}
\right)  \alpha^{k-i}a^{i}\right\Vert \right\} \\
&  =\left\vert \alpha\right\vert ^{k}=\max\left\{  \left\vert \alpha
\right\vert ^{k},\left\Vert a\right\Vert ^{k}\right\}  =\left\Vert \left(
\alpha,a\right)  \right\Vert ^{k}%
\end{align*}

\begin{definition}
An commutative Banach algebra $\mathcal{A}$ with unit is called a C-algebra if
there exists a locally compact space $X$ such that $\mathcal{A}$ is
isometrically isomorphic to $C_{\infty}\left(  X\right)  ,$ where $C_{\infty
}\left(  X\right)  $ is the space of all continuous functions from $X$ into
$\mathbb{K}$ which vanishes at infinity.
\end{definition}

As we know $\left\{  e_{j}=\left(  \delta_{i,j}\right)  _{i\in\mathbb{N}}%
:j\in\mathbb{N}\right\}  $, where $\delta_{i,j}$ denotes the Kronecker symbol,
is the canonical basis of $c_{0}\,.$ Since%
\[
e_{j}^{2}=e_{j}%
\]
and%
\[
\overline{\left\langle \left\{  e_{j}:j\in\mathbb{N}\right\}  \right\rangle
}=c_{0}%
\]
we conclude that the collection of all the idempotent elements of $c_{0}$ with
norm less than 1 is dense in $c_{0}.$ As a consequence, $c_{0}$ is a C-algebra
(Th. 6.12 $\left[  6\right]  $).

\begin{theorem}
$\mathcal{S}_{Y}\left(  c_{0}\right)  $ is a C-algebra.

\begin{proof}
It follows from the fact that $c_{0}^{+}$ is isometrically isomorphic to
$\mathcal{S}_{Y}\left(  c_{0}\right)  $.
\end{proof}
\end{theorem}

\begin{remark}
We recall that the spectrum of a commutative Banach algebra $\mathfrak{A}$ is
the collection $Sp\left(  \mathfrak{A}\right)  $ of all non-null homomorphisms
defined from $\mathfrak{A}$ into $\mathbb{K}$, that is,
\[
Sp\left(  \mathfrak{A}\right)  =\left\{  \phi:\mathfrak{A}\rightarrow
\mathbb{K}:\phi\text{ is a non-null homomorphism}\right\}  .
\]
Note that the natural topology in $Sp\left(  \mathfrak{A}\right)  $ is induced
by the product topology in $\mathbb{K}^{\mathfrak{A}}$ and also for each
$\phi\in Sp\left(  \mathfrak{A}\right)  ,$ $\left\Vert \phi\right\Vert \leq1.$
For any $x\in\mathfrak{A},$ we define%
\[
G_{x}:Sp\left(  \mathfrak{A}\right)  \rightarrow\mathbb{K},\ \ \phi\rightarrow
G_{x}\left(  \phi\right)  =\phi\left(  x\right)  .
\]
which is clearly continuous and bounded. Let us denote by
\[
\left\Vert x\right\Vert _{sp}=\sup\left\{  \left\vert \phi\left(  x\right)
\right\vert :\phi\in Sp\left(  \mathfrak{A}\right)  \right\}  =\left\Vert
G_{x}\right\Vert _{\infty}%
\]
the spectral norm of $x.$ Since
\[
\left\vert \phi\left(  x\right)  \right\vert \leq\left\Vert \phi\right\Vert
\left\Vert x\right\Vert \leq\left\Vert x\right\Vert ,
\]
we have, in general, that
\[
\left\Vert x\right\Vert _{sp}\leq\left\Vert x\right\Vert .
\]
Finally, let us denote by $R\left(  G_{x}\right)  $ the range of $G_{x}.$ The
closure of $R\left(  G_{x}\right)  ,$ $\overline{R\left(  G_{x}\right)  },$ is
called spectrum of $x.$
\end{remark}

L. Narici (Cor. 6.16, $\left[  6\right]  $) proved the following result:

\begin{proposition}
A commutative Banach algebra $\mathfrak{A}$ with unit is a $C-$algebra if and
only if its spectrum $Sp\left(  \mathfrak{A}\right)  $ is compact and its
spectral norm $\left\Vert x\right\Vert _{sp}$ is equal to $\left\Vert
x\right\Vert $, for every $x\in\mathfrak{A}.$
\end{proposition}

\begin{remark}
As a consequence of this corollary, we have that $Sp\left(  \mathcal{S}%
_{Y}\left(  c_{0}\right)  \right)  $ is compact and, for every $S\in
\mathcal{S}_{Y}\left(  c_{0}\right)  ,$
\[
\left\Vert S\right\Vert =\sup_{i\in\mathbb{N}}\left\Vert S\left(
e_{i}\right)  \right\Vert =\left\Vert S\right\Vert _{sp}.
\]
The compactness of $Sp\left(  \mathcal{S}_{Y}\left(  c_{0}\right)  \right)  $
guarantees that the spectrum of $S$ is compact in $\mathbb{K}$, since
$\mathcal{S}_{Y}\left(  c_{0}\right)  $ has a unit.
\end{remark}

\section{The subalgebra $\mathcal{L}_{T}$}

Let us fix a compact and self-adjoint operator $T$; hence $I+T\in
\mathcal{S}_{Y}\left(  c_{0}\right)  .$ We shall denote by $\mathcal{L}_{T}$
the closure of the algebra spanned by $\left\{  I,T\right\}  ,$ that is,
\[
\mathcal{L}_{T}=\overline{\left\langle \left\{  I,T\right\}  \right\rangle
}=\left\{  \sum_{n=0}^{\infty}\alpha_{n}T^{n}:\left\Vert \alpha_{n}%
T^{n}\right\Vert \rightarrow0\right\}  .
\]
Clearly, $\mathcal{L}_{T}$ is a C-algebra since it is closed Banach subalgebra
of $\mathcal{S}_{Y}\left(  c_{0}\right)  $ (Cor. 6.13, $\left[  6\right]  $).

This condition guarantees that $Sp\left(  \mathcal{L}_{T}\right)  $ is compact
and, for each $H\in\mathcal{L}_{T},$%
\[
\left\Vert H\right\Vert =\sup_{i\in I}\left\Vert H\left(  e_{i}\right)
\right\Vert =\left\Vert H\right\Vert _{sp}.
\]
On the other hand, since the operators norm in $\mathcal{S}_{Y}\left(
c_{0}\right)  $ is power multiplicative, such property is inherited by
$\mathcal{L}_{T}.$ Thus, for any $H\in\mathcal{L}_{T},$ we have
\[
\left\Vert H^{n}\right\Vert =\left\Vert H\right\Vert ^{n}.
\]

Under the conditions that $\mathcal{L}_{T}$ is a C-algebra and $Sp\left(
\mathcal{L}_{T}\right)  $ is compact, we conclude that $\mathcal{L}_{T}$ is
isometrically isomorphic to the space of all continuous functions $C\left(
Sp\left(  \mathcal{L}_{T}\right)  \right)  $ provided by the supremum norm,
that is, there exists an isomorphism of algebras $\Psi$ which is, at the same
time, an isometry from $\mathcal{L}_{T}$ onto $C\left(  Sp\left(
\mathcal{L}_{T}\right)  \right)  .$ Thus, if $H=\sum_{n=0}^{\infty}\alpha
_{n}T^{n},$ then $\left\Vert H\right\Vert =\left\Vert \Psi\left(  H\right)
\right\Vert _{\infty}.$

Now, since $T$ is compact and self-adjoint$,$ there exists $\lambda=\left(
\lambda_{i}\right)  _{i\in\mathbb{N}}\in c_{0}$ for which
\[
T=T_{\lambda}=\sum_{n=1}^{\infty}\lambda_{i}P_{i}\text{, }\left\Vert
T\right\Vert =\left\Vert \lambda\right\Vert \text{ and }T\left(  y^{\left(
i\right)  }\right)  =\lambda_{i}y^{\left(  i\right)  }\text{ for}\ y^{\left(
i\right)  }\in Y
\]

Let us denote by $\sigma\left(  T\right)  =\left\{  \lambda_{0},\lambda
_{1},\ldots\right\}  $ with $\lambda_{0}=0,$ the collection of all eigenvalues
of $T.$ We define the homomorphism of algebra $\phi_{i}:\left\langle \left\{
I,T\right\}  \right\rangle \rightarrow\mathbb{K}$ by%
\[
\phi_{i}\left(  T\right)  =\lambda_{i}.
\]
Let $H=\sum_{n=0}^{k}\alpha_{n}T_{\lambda}^{n}\in\left\langle \left\{
I,T\right\}  \right\rangle .$ Since%
\begin{align*}
\left\vert \phi_{i}\left(  H\right)  \right\vert  &  =\left\vert \alpha
_{0}+\sum_{n=1}^{k}\alpha_{n}\lambda_{i}^{n}\right\vert \leq\max\left\{
\left\vert \alpha_{0}\right\vert ,\left\vert \sum_{n=1}^{k}\alpha_{n}%
\lambda_{i}^{n}\right\vert \right\} \\
&  \leq\max\left\{  \left\vert \alpha_{0}\right\vert ,\left\Vert \sum
_{n=1}^{k}\alpha_{n}\lambda^{n}\right\Vert \right\}  =\left\Vert \left(
\alpha_{0},\sum_{n=1}^{k}\alpha_{n}\lambda^{n}\right)  \right\Vert _{c_{0}%
^{+}}\\
&  =\left\Vert \alpha_{0}I+T_{\sum_{n=1}^{k}\alpha_{n}\lambda^{n}}\right\Vert
=\left\Vert \sum_{n=0}^{k}\alpha_{n}T_{\lambda}^{n}\right\Vert =\left\Vert
H\right\Vert
\end{align*}
we get that $\phi_{i}$ is continuous.

Since $\left\{  I,T,T^{2},\ldots\right\}  $ also generates the closed algebra
$\mathcal{L}_{T}$ and for $H=\sum_{n=0}^{\infty}\alpha_{n}T^{n}\in
\mathcal{L}_{T},$ we have%
\begin{align*}
\left\vert \alpha_{n}\lambda_{i}^{n}\right\vert  &  =\left\vert \alpha
_{n}\right\vert \left\vert \lambda_{i}\right\vert ^{n}\leq\left\vert
\alpha_{n}\right\vert \left\Vert \lambda\right\Vert ^{n}\\
&  =\left\vert \alpha_{n}\right\vert \left\Vert T\right\Vert ^{n}=\left\vert
\alpha_{n}\right\vert \left\Vert T^{n}\right\Vert =\left\Vert \alpha_{n}%
T^{n}\right\Vert \rightarrow0,
\end{align*}
and then $\phi_{i}$ can be continuously extended to $\mathcal{L}_{T}.$

From this, the following function%
\[
\sigma\left(  T\right)  \overset{\Gamma}{\rightarrow}Sp\left(  \mathcal{L}%
_{T}\right)  ;\ \lambda_{i}\rightarrow\Gamma\left(  \lambda_{i}\right)
=\phi_{i}.
\]
is well defined. We claim that $\Gamma$ is bijective. Clearly, it is an
injective function. To prove the surjective condition, we will start with the
following result:

\begin{theorem}
If $z\notin\sigma\left(  T\right)  ,$ then $zI-T$ is invertible in
$\mathcal{S}_{Y}\left(  c_{0}\right)  .$

\begin{proof}
For $y\in R\left(  zI-T\right)  ,$ there exists $x\in c_{0}$ such that%
\[
\left(  zI-T\right)  \left(  x\right)  =y
\]
Since $z\notin\sigma\left(  T\right)  ,$ we can solve this equation for $x$
and get%
\begin{equation}
x=\frac{1}{z}y+\frac{1}{z}Tx=\frac{1}{z}y+\frac{1}{z}\sum_{i=1}^{\infty
}\lambda_{i}\frac{\left\langle x,y^{\left(  i\right)  }\right\rangle
}{\left\langle y^{\left(  i\right)  },y^{\left(  i\right)  }\right\rangle
}y^{\left(  i\right)  }%
\end{equation}
Applying the continuous functional $\left\langle \cdot,y^{\left(  k\right)
}\right\rangle $ above, we have%
\begin{align*}
\left\langle x,y^{\left(  k\right)  }\right\rangle  &  =\left\langle \frac
{1}{z}y+\frac{1}{z}\sum_{i=1}^{\infty}\lambda_{i}\frac{\left\langle
x,y^{\left(  i\right)  }\right\rangle }{\left\langle y^{\left(  i\right)
},y^{\left(  i\right)  }\right\rangle }y^{\left(  i\right)  },y^{\left(
k\right)  }\right\rangle \\
&  =\frac{1}{z}\left\langle y,y^{\left(  k\right)  }\right\rangle +\frac{1}%
{z}\frac{\lambda_{k}}{z}\left\langle x,y^{\left(  k\right)  }\right\rangle
\end{align*}
Now, solving the last equation for $\left\langle x,y^{\left(  k\right)
}\right\rangle ,$ we obtain%
\begin{align*}
\left(  1-\frac{\lambda_{k}}{z}\right)  \left\langle x,y^{\left(  k\right)
}\right\rangle  &  =\frac{1}{z}\left\langle y,y^{\left(  k\right)
}\right\rangle \\
\left\langle x,y^{\left(  k\right)  }\right\rangle  &  =\frac{1}{z-\lambda
_{k}}\left\langle y,y^{\left(  k\right)  }\right\rangle
\end{align*}
Note that the sequence%
\[
\left(  \frac{\lambda_{k}}{z-\lambda_{k}}\right)  _{k\in\mathbb{N}}%
\]
is an element of $c_{0}.$ In fact, since $z\notin\sigma\left(  T\right)  ,$ we
have that $\left\vert z\right\vert >0$ and, therefore, for a given
$0<\epsilon<1$, there exists $i_{0}\in\mathbb{N}\,$\ such that%
\[
i\geq i_{0}\Rightarrow\left\vert \lambda_{i}\right\vert <\epsilon\left\vert
z\right\vert .
\]
Thus,%
\[
i\geq i_{0}\Rightarrow\left\vert \frac{\lambda_{i}}{z-\lambda_{i}}\right\vert
=\frac{\left\vert \lambda_{i}\right\vert }{\left\vert z\right\vert }%
<\epsilon.
\]
Now, replacing in $\left(  3.1\right)  ,$ we get%
\begin{align*}
x  &  =\frac{1}{z}y+\frac{1}{z}\sum_{i=1}^{\infty}\frac{\lambda_{i}}%
{z-\lambda_{i}}\frac{\left\langle y,y^{\left(  i\right)  }\right\rangle
}{\left\langle y^{\left(  i\right)  },y^{\left(  i\right)  }\right\rangle
}y^{\left(  i\right)  }\\
&  =\frac{1}{z}y+\frac{1}{z}\sum_{i=1}^{\infty}\frac{\lambda_{i}}%
{z-\lambda_{i}}P_{i}\left(  y\right)
\end{align*}
Although $y$ belongs to $R\left(  zI-T\right)  ,$ the last expression holds
for any $y\in c_{0}.$ Thus, if we denote by
\[
R_{z}\left(  T\right)  \left(  y\right)  =\frac{1}{z}y+\frac{1}{z}\sum
_{i=1}^{\infty}\frac{\lambda_{i}}{z-\lambda_{i}}P_{i}\left(  y\right)  ,
\]
then $R_{z}\left(  T\right)  \left(  \cdot\right)  \in\mathcal{S}_{Y}\left(
c_{0}\right)  ,$ since $\sum_{i=1}^{\infty}\frac{\lambda_{i}}{z-\lambda_{i}%
}P_{i}\left(  \cdot\right)  $ is compact and self-adjoint operator.\newline
Let us show that , effectively, $R_{z}\left(  T\right)  \left(  \cdot\right)
$ is the inverse operator of $zI-T:$%
\begin{align*}
&  \left.  \left[  \left(  zI-T\right)  \circ R_{z}\left(  T\right)  \right]
\left(  y\right)  =\right.  \left(  zI-T\right)  \left(  \frac{1}{z}y+\frac
{1}{z}\sum_{i=1}^{\infty}\frac{\lambda_{i}}{z-\lambda_{i}}P_{i}\left(
y\right)  \right) \\
&  =y+\sum_{i=1}^{\infty}\frac{\lambda_{i}}{z-\lambda_{i}}P_{i}\left(
y\right)  -\frac{1}{z}\sum_{i=1}^{\infty}\lambda_{i}P_{i}\left(  y\right)
-\frac{1}{z}\sum_{i=1}^{\infty}\frac{\lambda_{i}^{2}}{z-\lambda_{i}}%
P_{i}\left(  y\right)  ;\ \ T\left(  P_{i}\left(  y\right)  \right)
=\lambda_{i}P_{i}\left(  y\right) \\
&  =y+\sum_{i=1}^{\infty}\left[  \underset{=0}{\underbrace{\frac{\lambda_{i}%
}{z-\lambda_{i}}-\frac{\lambda_{i}}{z}-\frac{\lambda_{i}^{2}}{z\left(
z-\lambda_{i}\right)  }}}\right]  P_{i}\left(  y\right)  =y=I\left(  y\right)
\end{align*}
In the other direction, since $P_{j}\left(  P_{i}\left(  x\right)  \right)
=P_{i}\left(  x\right)  ,$ if $j=i,$ otherwise it is $0,$ we have
\begin{align*}
&  \left.  \left[  R_{z}\left(  T\right)  \circ\left(  zI-T\right)  \right]
\left(  x\right)  =\right.  zR_{z}\left(  T\right)  \left(  x\right)
-R_{z}\left(  T\right)  \left(  Tx\right) \\
&  =x+\sum_{i=1}^{\infty}\frac{\lambda_{i}}{z-\lambda_{i}}P_{i}\left(
x\right)  -\sum_{i=1}^{\infty}\lambda_{i}R_{z}\left(  T\right)  \left(
P_{i}\left(  x\right)  \right) \\
&  =x+\sum_{i=1}^{\infty}\frac{\lambda_{i}}{z-\lambda_{i}}P_{i}\left(
x\right)  -\sum_{i=1}^{\infty}\lambda_{i}\left[  \frac{1}{z}P_{i}\left(
x\right)  +\frac{1}{z}\sum_{j=1}^{\infty}\frac{\lambda_{j}}{z-\lambda_{j}%
}P_{j}\left(  P_{i}\left(  x\right)  \right)  \right] \\
&  =x+\sum_{i=1}^{\infty}\left[  \underset{=0}{\underbrace{\frac{\lambda_{i}%
}{z-\lambda_{i}}-\frac{\lambda_{i}}{z}-\frac{\lambda_{i}^{2}}{z\left(
z-\lambda_{i}\right)  }}}\right]  P_{i}\left(  x\right)  =x=I\left(  x\right)
\end{align*}
Therefore, $R_{z}\left(  T\right)  =\left(  zI-T\right)  ^{-1}\in
\mathcal{S}_{Y}\left(  c_{0}\right)  .$
\end{proof}
\end{theorem}

\begin{corollary}
If $z\notin\sigma\left(  T\right)  ,$ then $R_{z}\left(  T\right)
\in\mathcal{L}_{T}$.

\begin{proof}
We already know that $\mathcal{L}_{T}$ is a C-algebra with unit; hence, by Th.
6.10 in $\left[  6\right]  $, we have that
\[
\left(  zI-T\right)  ^{-1}\in\overline{\mathbb{K}\left[  zI-T\right]
}=\overline{\left\langle \left\{  zI-T,\left(  zI-T\right)  ^{2},\left(
zI-T\right)  ^{3},\cdots\right\}  \right\rangle }%
\]
Now, since $\left(  zI-T\right)  ^{n}\in\mathcal{L}_{T}$ for any
$n\in\mathbb{N},$ we have that $\overline{\mathbb{K}\left[  zI-T\right]  }$ is
a subalgebra of $\mathcal{L}_{T}.$ This proves that $R_{z}\left(  T\right)
\in\mathcal{L}_{T}.$
\end{proof}
\end{corollary}

\begin{corollary}
The function $\sigma\left(  T\right)  \overset{\Gamma}{\rightarrow}Sp\left(
\mathcal{L}_{T}\right)  $ is bijective.

\begin{proof}
We already know that $\Gamma$ is injective. If $\phi\in Sp\left(
\mathcal{L}_{T}\right)  ,$ then $\phi\left(  T\right)  =z,$ for a
$z\in\mathbb{K}.$ Suppose that $z\notin\sigma\left(  T\right)  ,$ hence $zI-T$
has an inverse and, for the previous theorem, $\left(  zI-T\right)
^{-1}=R_{z}\left(  T\right)  \in\mathcal{L}_{T}.$ Since the function $\phi$ is
a homomorphism of algebras with unit, we have%
\[
1=\phi\left(  I\right)  =\phi\left(  \left(  zI-T\right)  ^{-1}\circ\left(
zI-T\right)  \right)  =\phi\left(  \left(  zI-T\right)  ^{-1}\right)
\phi\left(  zI-T\right)  ,
\]
but, by the linearity of $\phi,$ the factor $\phi\left(  zI-T\right)  $ is
null$,$ which is a contradiction. Such contradiction is coming from the fact
that $z\notin\sigma\left(  T\right)  .$\newline Thus, if $\phi\in Sp\left(
\mathcal{L}_{T}\right)  ,$ then there exists $\mu\in\sigma\left(  T\right)  $
such that $\phi=\phi_{\mu}$ and therefore $\Gamma$ is bijective.
\end{proof}
\end{corollary}

\begin{remark}
We already have identified $Sp\left(  \mathcal{L}_{T}\right)  $ with
$\sigma\left(  T\right)  $ through the bijective function $\Gamma$. Let us
consider the subspace topology of $\mathbb{K}$ on $\sigma\left(  T\right)  .$
Also, note that $\sigma\left(  T\right)  $ is compact with the subspace
topology. \newline Now, we affirm that $\Upsilon=\Gamma^{-1}$ is continuous.
In fact, if $\phi_{\alpha}\rightarrow\phi$ in the induced topology by the
product topology in $\mathbb{K}^{\mathcal{L}_{T}},$ then%
\[
\phi_{\alpha}\left(  H\right)  \rightarrow\phi\left(  H\right)
\]
for all $H\in\mathcal{L}_{T},$ in particular,
\[
\phi_{\alpha}\left(  T\right)  \rightarrow\phi\left(  T\right)
\]
or, equivalently,
\[
\Upsilon\left(  \phi_{\alpha}\right)  \rightarrow\Upsilon\left(  \phi\right)
\]
Now, since $\Upsilon$ is bijective and continuous, $Sp\left(  \mathcal{L}%
_{T}\right)  $ is compact and $\sigma\left(  T\right)  $ is a Hausdorff space,
we conclude that $\Upsilon$ is a homeomorphism.\newline By these facts and by
the uniqueness of $X$ (up to homeomorphism) for which $\mathcal{L}_{T}\cong
C\left(  X\right)  ,$ we have
\[
\mathcal{L}_{T}\cong C\left(  \sigma\left(  T\right)  \right)  .
\]

\end{remark}

From the fact that $Sp\left(  \mathcal{L}_{T}\right)  $ is homeomorphic to
$\sigma\left(  T\right)  ,$
\[
G_{T}:\sigma\left(  T\right)  \rightarrow\mathbb{K};\ \lambda_{i}\rightarrow
G_{T}\left(  \lambda_{i}\right)  =\lambda_{i}\
\]
in other words, $G_{T}=f_{T}$ is the identity map. Thus, if $H=\alpha
_{0}I+\sum_{n=1}^{\infty}\alpha_{n}T^{n}\in\mathcal{L}_{T},$ then
\[
G_{H}:\sigma\left(  T\right)  \rightarrow\mathbb{K};\ \lambda_{i}\rightarrow
G_{H}\left(  \lambda_{i}\right)  =\alpha_{0}+\sum_{n=1}^{\infty}\alpha
_{n}\lambda_{i}^{n}.
\]

Therefore, we can get the well-known Gelfand transformation
\[
G:\mathcal{L}_{T}\rightarrow C\left(  \sigma\left(  T\right)  \right)
;\ H\rightarrow G_{H}%
\]

\begin{proposition}
$G$ is an isometric isomorphism algebra.

\begin{proof}
Clearly, $G$ is an algebra homomorphism. Since $\mathcal{L}_{T}$ satisfies the
condition given by Cor. 6.16, pp. 218, $\left[  6\right]  ,$ we have%
\[
\left\Vert G_{H}\right\Vert _{\infty}=\left\Vert H\right\Vert _{sp}=\left\Vert
H\right\Vert .
\]
Thus, it is enough to prove that $G$ is surjective. By the fact that $G$ is an
algebra homomorphism and the image of $T$ by $G$ is the identity map
$G_{T}=f_{T}$, the collection $\left\{  1,f_{T},f_{T}^{2},\ldots\right\}  $ is
the image of $\left\{  I,T,T^{2},\ldots\right\}  .$ \newline Now, since
$\sigma\left(  T\right)  $ is compact, Kaplanski (Th. 5.28, pp. 191 $\left[
6\right]  $) guarantees that $\left[  \left\{  1,f_{T},f_{T}^{2}%
,\ldots\right\}  \right]  $ is dense in $C\left(  \sigma\left(  T\right)
\right)  .$ Thus, if $f\in C\left(  \sigma\left(  T\right)  \right)  ,$ then
there exists a sequence $\left\{  g_{n}\right\}  _{n\in\mathbb{N}}$ in
$\left[  \left\{  1,f_{T},f_{T}^{2},\ldots\right\}  \right]  $ such that
$f=\lim_{n\rightarrow\infty}g_{n}.$ Now, for each $n\in\mathbb{N},$ there
exists $H_{n}\in\mathcal{L}_{T}$ such that%
\[
G_{H_{n}}=g_{n}.
\]
By the fact that $G$ is an isometry$,$ the sequence $\left(  H_{n}\right)  $
is a Cauchy sequence in $\mathcal{L}_{T}$. Let us denote by $H=\lim
_{n\rightarrow\infty}H_{n}.$ Since $G$ is continuous, we have that
\[
G_{H}=\lim_{n\rightarrow\infty}G_{H_{n}}=\lim_{n\rightarrow\infty}g_{n}=f.
\]

\end{proof}
\end{proposition}

\section{An integral.}

By the previous section, there exists an algebra isometric isomorphism
\break$\Psi=G^{-1}:C\left(  \sigma\left(  T\right)  \right)  \rightarrow
\mathcal{L}_{T}.$ Let us denote by $\Omega\left(  \sigma\left(  T\right)
\right)  $ the Boolean ring of all clopen subsets of $\sigma\left(  T\right)
.$

For a $C\subset\sigma\left(  T\right)  ,$ the function $\eta_{C}$ denotes the
characteristic function of $C.$ If $C_{1},C_{2}\subset\sigma\left(  T\right)
,$ then%
\begin{align*}
\eta_{C_{1}}\cdot\eta_{C_{2}}  &  =\eta_{C_{1}\cap C_{2}};\ \ \eta_{C}%
^{2}=\eta_{C}\\
\eta_{C_{1}}+\eta_{C_{2}}  &  =\eta_{C_{1}\cup C_{2}},\text{ if }C_{1}\cap
C_{2}=\varnothing.
\end{align*}
Of course, $\eta_{C}$ is continuous if, and only if, $C\in\Omega\left(
\sigma\left(  T\right)  \right)  .$

Now, since $\Psi$ is a homomorphism of algebras, we have
\[
\Psi\left(  \eta_{C}\right)  =\Psi\left(  \eta_{C}^{2}\right)  =\Psi\left(
\eta_{C}\right)  ^{2}.
\]
In other words, $\Psi\left(  \eta_{C}\right)  $ is a projection, even more, if
$C\in\Omega\left(  \sigma\left(  T\right)  \right)  \setminus\left\{
\varnothing\right\}  ,$ then $\Psi\left(  \eta_{C}\right)  $ is a non-null
projection in $\mathcal{L}_{T}.$

On the other hand, we know that the linear hull of $\left\{  \eta_{C}%
:C\in\Omega\left(  \sigma\left(  T\right)  \right)  \right\}  $ is dense in
$C\left(  \sigma\left(  T\right)  \right)  ;$ hence, for a given $\epsilon>0$
and for $f\in C\left(  \sigma\left(  T\right)  \right)  ,$ there exists a
finite clopen partition $\left\{  C_{k}:k=1,\ldots,n\right\}  $ of
$\sigma\left(  T\right)  $ and a finite collection of scalars $\left\{
\alpha_{k}:k=1,\ldots,n\right\}  $ such that
\begin{equation}
\left\Vert f-\sum_{k=1}^{n}\alpha_{k}\eta_{C_{k}}\right\Vert _{\infty}%
=\sup_{x\in\sigma\left(  T\right)  }\left\vert f\left(  x\right)  -\sum
_{k=1}^{n}\alpha_{k}\eta_{C_{k}}\left(  x\right)  \right\vert <\epsilon
\end{equation}
Without lost of generality, we can assume that, for each $k=1,\ldots,n,$ there
exists $x_{k}\in\sigma\left(  T\right)  $ such that
\[
\left\Vert f-\sum_{k=1}^{n}\alpha_{k}\eta_{C_{k}}\right\Vert _{\infty}%
=\sup_{x\in\sigma\left(  T\right)  }\left\vert f\left(  x\right)  -\sum
_{k=1}^{n}f\left(  x_{k}\right)  \eta_{C_{k}}\left(  x\right)  \right\vert
<\epsilon
\]
Using the isometry of $\Psi,$ we have%
\[
\left\Vert \Psi\left(  f\right)  -\sum_{k=1}^{n}f\left(  x_{k}\right)
\Psi\left(  \eta_{C_{k}}\right)  \right\Vert <\epsilon
\]
If we denote by $E_{C_{k}}$ the corresponding projection $\Psi\left(
\eta_{C_{k}}\right)  ,$ then
\[
\left\Vert \Psi\left(  f\right)  -\sum_{k=1}^{n}f\left(  x_{k}\right)
E_{C_{k}}\right\Vert <\epsilon
\]
This also shows that the space%
\[
\left\langle \left\{  E\in\mathcal{L}_{T}:E^{2}=E\right\}  \right\rangle
\]
is dense in $\mathcal{L}_{T}.$

Note that these subsets can be classified as follow: the first type (or type
1) are those finite subsets of $\left\{  \lambda_{n}:n\geq1\right\}  $ and the
second type (or type 2) are those subsets of the form $\sigma\left(  T\right)
\setminus C$ where $C$ is a finite subset of $\left\{  \lambda_{n}%
:n\geq1\right\}  .$

Let us consider the following set-function:%
\[
m_{T}:\Omega\left(  \sigma\left(  T\right)  \right)  \rightarrow
\mathcal{L}_{T};\ \ C\rightarrow m_{T}\left(  C\right)  =\Psi\left(  \eta
_{C}\right)  =E_{C}.
\]
Note that if $C=\left\{  \lambda_{i_{1}},\lambda_{i_{2}},\ldots,\lambda
_{i_{n}}\right\}  \in\Omega\left(  \sigma\left(  T\right)  \right)  ,$ then
\[
m_{T}\left(  C\right)  =\sum{}_{k=1}^{n}E_{i_{k}},\text{\ where }E_{i_{k}%
}=E_{\left\{  \lambda_{i_{k}}\right\}  }%
\]
or, if $C=\sigma\left(  T\right)  \setminus\left\{  \lambda_{i_{1}}%
,\lambda_{i_{2}},\ldots,\lambda_{i_{n}}\right\}  ,$ then
\[
m_{T}\left(  C\right)  =I-\sum{}_{k=1}^{n}E_{i_{k}}.
\]
On the other hand, if $\mathcal{C}=\left\{  C_{\mu}\right\}  _{\mu\in I}$ is
shrinking of $\Omega\left(  \sigma\left(  T\right)  \right)  $ and $\cap
_{\mu\in I}C_{\mu}=\varnothing,$ then there exists $\mu_{0}\in I$ such that
for $\mu\geq$ $\mu_{0},$ $C_{\mu}=\varnothing.$

Therefore, we can prove that $m_{T}$ satisfies:

\begin{enumerate}
\item If $\left\{  C_{k}:k=1,\ldots,n\right\}  \subset\Omega\left(
\sigma\left(  T\right)  \right)  $ such that $C_{h}\cap C_{k}=\varnothing$ for
$h\neq k$, then
\[
m_{T}\left(  \cup_{k=1}^{n}C_{k}\right)  =\sum{}_{k=1}^{n}m_{T}\left(
C_{k}\right)  .
\]

\item For all $C\in\Omega\left(  \sigma\left(  T\right)  \right)  ,\ \left\{
m_{T}\left(  B\right)  :B\in\Omega\left(  \sigma\left(  T\right)  \right)
,B\subset C\right\}  $ is bounded.

\item If $\mathcal{C}$ is a collection in $\Omega\left(  \sigma\left(
T\right)  \right)  $ is shrinking and $\cap_{C\in\mathcal{C}}C=\varnothing,$
then \break$\lim_{C\in\mathcal{C}}m_{T}\left(  C\right)  =\theta$.
\end{enumerate}

that is, $m_{T}$ is a finite additive measure.

Take a $C\in\Omega\left(  \sigma\left(  T\right)  \right)  ,\ C\neq
\varnothing$ and denote by $\mathcal{D}_{C}$ the collection of all
\break$\alpha=\left\{  C_{1},C_{2},\ldots,C_{n};x_{1},x_{2},\ldots
,x_{n}\right\}  ,$ where $\left\{  C_{k}:k=1,\ldots,n\right\}  $ is a clopen
partition of $C$ and $x_{k}\in C_{k}.$ We define an order by $\alpha_{1}%
\geq\alpha_{2}$ if, and only if, the clopen partition of $C$ in $\alpha_{1}$
is a refinement of the clopen partition of $C$ in $\alpha_{2}.$ Thus,
$\mathcal{D}_{C},$ with this order, is a directed set.

Now, if $f\in C\left(  \sigma\left(  T\right)  \right)  $ and $\alpha=\left\{
C_{1},C_{2},\ldots,C_{n};x_{1},x_{2},\ldots,x_{n}\right\}  \in\mathcal{D}%
_{C},$ then we define%
\[
\omega_{\alpha}(f,m_{T},C)=\sum_{k=1}^{n}f\left(  x_{k}\right)  m_{T}\left(
C_{k}\right)  =\sum_{k=1}^{n}f\left(  x_{k}\right)  E_{C_{k}}.
\]
For each $f\in C\left(  \sigma\left(  T\right)  \right)  $ and $C\in
\Omega\left(  \sigma\left(  T\right)  \right)  ,$ the continuous function
$f\eta_{C}$ can be reached by a net in $C\left(  \sigma\left(  T\right)
\right)  $ of type%
\[
\left\{  \sum_{k=1}^{n}f\left(  x_{k}\right)  \eta_{C_{k}}\right\}
_{\alpha\in\mathcal{D}_{C}}%
\]
By the isometry of $\Psi,$%
\[
\lim_{\alpha\in\mathcal{D}_{C}}\omega_{\alpha}(f,m_{T},C)=\Psi\left(
f\eta_{C}\right)
\]
Therefore, the operator $\Psi\left(  f\right)  $ can be interpreted as an
integral which we will denote by
\[
\Psi\left(  f\eta_{C}\right)  =\int_{\sigma\left(  T\right)  }f\eta_{C}%
dm_{T}=\int_{C}fdm_{T}=\lim_{\alpha\in\mathcal{D}_{C}}\omega_{\alpha}%
(f,m_{T},C)
\]
As a particular cases are $T=\Psi\left(  f_{T}\right)  $ and $I=\Psi\left(
1\right)  $
\[
T=\Psi\left(  f_{T}\right)  =\int_{\sigma\left(  T\right)  }f_{T}%
dm_{T};\ \ I=\Psi\left(  1\right)  =\int_{\sigma\left(  T\right)  }dm_{T}.\
\]

\section{Finite range self-adjoint operators}

In this section we will suppose that there exists an $n\in\mathbb{N}$ such
that $\lambda_{k}=0$, for $k\geq n+1$. Of course, the operator Sea
$T=\sum\limits_{i=1}^{\infty}\lambda_{i}P_{i}=\sum\limits_{i=1}^{n}\lambda
_{i}P_{i}$ and it is an operator of finite range.

Note that, in this case, $\sigma\left(  T\right)  =\left\{  \lambda
_{0},\lambda_{1},....,\lambda_{n}\right\}  \text{ with }\lambda_{0}=0$

As before,
\[
\mathcal{L}_{T}=\left\{  \alpha_{0}I+\sum\limits_{j=1}^{\infty}\alpha_{j}%
T^{j}\in\mathcal{L}\left(  c_{0}\right)  :\lim_{j\rightarrow\infty}\alpha
_{j}T^{j}=0\right\}
\]

Since $T^{j}=\sum\limits_{i=1}^{n}\lambda_{i}^{j}P_{i}\text{ and}$%

\[
\sum\limits_{j=1}^{\infty}\alpha_{j}T^{j}=\sum\limits_{j=1}^{\infty}%
\sum\limits_{i=1}^{n}\alpha_{j}\lambda_{i}^{j}P_{i}=\sum\limits_{i=1}%
^{n}\left(  \sum\limits_{j=1}^{\infty}\alpha_{j}\lambda_{i}^{j}\right)
P_{i},
\]

we conclude that%
\[
\mathcal{L}_{T}\subset\left[  I,P_{1},...,P_{n}\right]
\]

\begin{theorem}
Suppose that $\lambda_{i}\neq\lambda_{j}$ $\forall i,j=1,2,...,n$ with $i\neq
j$. Then,%
\[
\mathcal{L}_{T}=[\{I,P_{1},...,P_{n}\}]
\]

\end{theorem}

\begin{proof}
Using the Van der Monde method to calculate determinants whose columns are
geometric progressions, we have%
\[
\left\vert
\begin{array}
[c]{cccc}%
\lambda_{1} & \lambda_{2} & ... & \lambda_{n}\\
\lambda_{1}^{2} & \lambda_{2}^{2} & ... & \lambda_{n}^{2}\\%
\begin{array}
[c]{c}%
.\\
.
\end{array}
&
\begin{array}
[c]{c}%
.\\
.
\end{array}
& \cdots &
\begin{array}
[c]{c}%
.\\
.
\end{array}
\\
\lambda_{1}^{n} & \lambda_{2}^{n} & ... & \lambda_{n}^{n}%
\end{array}
\right\vert =\left(  \underset{i=1}{\overset{n}{\Pi}}\lambda_{i}\right)
\left(  \underset{1\leq i<j\leq n}{\Pi}\left(  \lambda_{j}-\lambda_{i}\right)
\right)  \neq0
\]
since $\lambda_{i}\not =\lambda_{j}$ \ $\forall i\not =j$.

Considering the equation system%
\begin{align*}
T  &  =\lambda_{1}P_{1}+\lambda_{2}P_{2}+....+\lambda_{n}P_{n}\\
T^{2}  &  =\lambda_{1}^{2}P_{1}+\lambda_{2}^{2}P_{2}+...+\lambda_{n}^{2}%
P_{n}\\
&  ...\\
T^{n}  &  =\lambda_{1}^{n}P_{1}+\lambda_{2}^{n}P_{2}+...+\lambda_{n}^{n}P_{n}%
\end{align*}
and applying the Van der Monde method for operators, we obtain%
\[
P_{k}=\underset{i=0,i\neq k}{\overset{n}{\Pi}}\left(  \frac{\lambda_{i}%
I-T}{\lambda_{i}-\lambda_{k}}\right)  \text{ }k=1,2,...,n.
\]
Therefore, $P_{1},...,P_{n}\in\mathcal{L}_{T}$ and, as a consequences,
$\mathcal{L}_{T}=\left[  \left\{  I,P_{1},P_{2},...,P_{n}\right\}  \right]  .$
\end{proof}

\begin{remark}
On the other hand, suppose, for instance, that $\lambda_{1}=\lambda_{2}$ and
\ $\lambda_{i}\neq\lambda_{j}$ for $i,j=2,...,n$ with $i\neq j$. Then,
\[
T=\lambda_{1}\left(  P_{1}+P_{2}\right)  +\lambda_{3}P_{3}+...+\lambda
_{n}P_{n}.
\]
Using the same arguments as in the previous theorem, we get that
\[
\mathcal{L}_{T}=\left[  I,P_{1}+P_{2},P_{3},...,P_{n}\right]  .
\]
Note that $P_{1}\notin\mathcal{L}_{T}$. In fact, if
\begin{equation}
P_{1}=\alpha_{0}I+\alpha_{1}(P_{1}+P_{2})+...+\alpha_{n-1}P_{n},
\end{equation}
then%
\[
P_{1}=P_{1}P_{1}=\left(  \alpha_{0}I+\alpha_{1}(P_{1}+P_{2})+...+\alpha
_{n-1}P_{n}\right)  P_{1}=\left(  \alpha_{0}+\alpha_{1}\right)  P_{1}%
\]
which implies that $\alpha_{0}+\alpha_{1}=1$. On the other hand, if we do the
something, but this time with $P_{2}$ in $\left(  5.1\right)  $, we get that
$\alpha_{0}+\alpha_{1}=0,$ which is a contradiction. In similar way, we can
prove that $P_{2}\notin\mathcal{L}_{T}$.\newline
\end{remark}

Observe that%
\[
\Psi\left(  \eta_{\left\{  \lambda_{k}\right\}  }\right)  =G^{-1}\left(
\eta_{\left\{  \lambda_{k}\right\}  }\right)  =P_{k},\ k\neq0
\]
In fact, the Gelfand transformation of $P_{k}$ is the following:
\begin{align*}
G_{P_{k}}\left(  \phi_{j}\right)   &  =\phi_{j}\left(  P_{k}\right)  =\phi
_{j}\left(  \underset{i=0,i\neq k}{\overset{n}{\Pi}}\left(  \frac{\lambda
_{i}-T}{\lambda_{i}-\lambda_{k}}\right)  \right) \\
&  =\underset{i=0,i\neq k}{\overset{n}{\Pi}}\left(  \frac{\phi_{j}\left(
\lambda_{i}I\right)  -\phi_{j}\left(  T\right)  }{\lambda_{i}-\lambda_{k}%
}\right)  =\underset{i=0,i\neq k}{\overset{n}{\Pi}}\left(  \frac{\lambda
_{i}-\lambda_{j}}{\lambda_{i}-\lambda_{k}}\right) \\
&  =\left\{
\begin{array}
[c]{ccc}%
1 & if & j=k\\
0 & if & j\neq k
\end{array}
\right.  =\eta_{\left\{  \lambda_{k}\right\}  }\left(  \lambda_{j}\right)
\end{align*}
If $k=0,$%
\[
\Psi\left(  \eta_{\left\{  \lambda_{0}\right\}  }\right)  =G^{-1}\left(
\eta_{\left\{  \lambda_{0}\right\}  }\right)  =I-\sum_{i=1}^{n}P_{i}%
\]

Note that the corresponding Boolean ring $\Omega\left(  \sigma\left(
T\right)  \right)  $ is the collection of all subset of $\sigma\left(
T\right)  $, $\ $and its measure
\[
m:\Omega\left(  \sigma\left(  T\right)  \right)  \rightarrow\mathcal{L}_{T}%
\]
is defined by $m\left(  \left\{  \lambda_{k}\right\}  \right)  =\Psi\left(
\eta_{\left\{  \lambda_{k}\right\}  }\right)  .$

If the eigenvalues are pairwise different, then$\ $%
\begin{align*}
m(\emptyset)  &  =0\\
m(\sigma(T))  &  =I\ \\
m(\{\lambda_{0}\})  &  =I-\sum_{i=1}^{n}P_{i};\ \\
m(\{\lambda_{i}\})  &  =P_{i}\ \text{for}\ i\neq0\\
\ m(\{\lambda_{0},\lambda_{i}\})  &  =I-\underset{1\leq j\leq n;\ j\neq
i}{\sum}P_{j}%
\end{align*}
\newline and for $f\in C\left(  \sigma(T),\mathbb{K}\right)  $
\[
\int_{\sigma\left(  T\right)  }fdm=f(\lambda_{0})I+\sum_{i=1}^{n}\left[
f(\lambda_{i})-f(\lambda_{0})\right]  P_{i}.
\]

On the other hand, if a couple of the eigenvalues are equal, say for example
$\lambda_{1}=\lambda_{2}$ and the rest of them are different, then
\begin{align*}
m(\emptyset)  &  =0\ \\
m(\sigma(T))  &  =I\\
m(\{\lambda_{0}\})  &  =I-(P_{1}+P_{2})-\sum_{i=3}^{n}P_{i}\\
m(\{\lambda_{1}\})  &  =P_{1}+P_{2}\\
m(\sigma(T)-\{\lambda_{1}\})  &  =I-(P_{1}+P_{2})\\
\ m(\{\lambda_{i}\})  &  =P_{i}\ \text{if }i\not =1,2\\
\ m(\{\lambda_{0},\lambda_{1}\})  &  =I-\sum_{j=3}^{n}P_{j}\\
m(\{\lambda_{0},\lambda_{i}\})  &  =I-(P_{1}+P_{2})-\sum_{j=3}^{n}%
P_{j}\ \text{if }i\not =1,2\ and\ j\not =i\newline%
\end{align*}
and for $f\in C\left(  \sigma(T),\mathbb{K}\right)  $
\[
\int_{\sigma\left(  T\right)  }fdm=f(\lambda_{0})I+\left[  f(\lambda
_{1})-f(\lambda_{0})\right]  (P_{1}+P_{2})+\sum_{i=3}^{n}\left[  f(\lambda
_{i})-f(\lambda_{0})\right]  P_{i}.
\]

\bigskip\bigskip

\end{document}